%% file: Transverse_ray_transform-3dim.tex
\documentclass[11pt]{article}
\usepackage[pagebackref,colorlinks=true,urlcolor=blue,linkcolor=blue,citecolor=blue]{hyperref}

\usepackage{amsthm, amssymb, bm}
\usepackage{soul, color}
\usepackage[]{amsmath}
\usepackage[]{amsfonts}
\usepackage[]{fancyhdr}
\usepackage[]{graphicx}
\graphicspath{{/EPSF/}{../figures/}{figures/}}

\input preamble

\newcommand{\T}{\mathcal{T}}

\newcommand{\D}{\mathrm{d}}

\newcommand{\lb}{\left(}

\newcommand{\rb}{\right)}

\newcommand{\OM}{\omega}
\newcommand{\wt}{\widetilde}
\newcommand{\Ac}{\mathcal{A}}
\newcommand{\Bc}{\mathcal{B}}
\newcommand{\Cc}{\mathcal{C}}

\newcommand{\Oc}{\mathcal{O}}

\newcommand{\Rc}{\mathcal{R}}

\newcommand{\Tc}{\mathcal{T}}

\newcommand{\Rb}{\mathbb{R}}
\newcommand{\Sb}{\mathbb{S}}

\newcommand{\g}{\boldsymbol{\gamma}}

\usepackage{amsthm}
\usepackage{amsfonts}
\newcommand{\Beq}{\begin{equation}}
\newcommand{\Eeq}{\end{equation}}
\newcommand{\beq}{\begin{equation*}}
\newcommand{\eeq}{\end{equation*}}
\newcommand{\bal}{\begin{align}}
\newcommand{\eal}{\end{align}}

\title{Microlocal inversion of a 3-dimensional restricted transverse ray transform of symmetric $m$-tensor fields}
\author{Venkateswaran P. Krishnan\thanks{Tata Institute of Fundamental Research, Centre For Applicable Mathematics, Bangalore, India. vkrishnan@math.tifrbng.res.in} \and Rohit Kumar Mishra\thanks{Department of Mathematics, University of California, Santa Cruz, CA 95064, USA. rokmishr@ucsc.edu} \and Suman Kumar Sahoo\thanks{Tata Institute of Fundamental Research, Centre For Applicable Mathematics, Bangalore, India. suman@tifrbng.res.in}}

\begin{document}
\maketitle
\begin{abstract}
We study the problem of inverting a  restricted transverse ray transform to recover a symmetric $m$-tensor field  in $\Rb^{3}$ using microlocal analysis techniques. More precisely, we prove that a symmetric $m$-tensor field can be recovered up to a known singular term and a smoothing term if its transverse ray transform is known along all lines intersecting a fixed smooth curve satisfying the Kirillov-Tuy condition.
\end{abstract}

	\section{Introduction}\label{sec:introduction}

The study of transverse ray transforms (TRT) of symmetric tensor fields is of interest in problems arising in polarization and diffraction tomography. % \cite{Sharafutdinov_Book,Lionheart-Withers}.
% and diffraction tomography \cite{Lionheart-Withers}.  
We are interested in an approximate inversion of a   TRT acting on symmetric tensor fields restricted to all lines passing through a fixed curve in $\Rb^{3}$. % This particular restricted TRT could be of potential interest in the diffraction tomography of strain recently investiga 
More precisely, we use techniques from microlocal analysis to construct a relative left parametrix for such restricted TRT. 

%We start by introducing some notation and define TRT. %The problem we study is the inversion of a restricted transverse ray transform which could be of potential interest in these tomographic applications. %
We denote the space of covariant symmetric $m$-tensors in $\Rb^{3}$ by $S^{m}=S^{m}(\Rb^{3})$. Let $C_{c}^{\infty}(S^{m})$ be the space of smooth compactly supported symmetric $m$-tensor fields in $\Rb^{3}$. In $\Rb^{3}$,  an element $f\in C_{c}^{\infty}(S^{m})$ can be written as 
% $m$-tensors on $\Rb^{n}$. Sections of this bundle are called symmetric $m$-tensor fields.  In Euclidean coordinates, any  symmetric $m$-tensor field can be written as 
% functions from $m$-Cartesian product $\Rb^{n}\times \cdots \times \Rb^{n} \to \Cb$, that is $\Rb$-linear and symmetric in each of the $m$ components. By abuse of notation, we also denote the space of symmetric tensor fields in $\Rb^{n}$ by $S^{m}(\Rb^{n})$. Any tensor field $f\in S^{m}(\Rb^{n})$ can be written as 
\[
f(x)=f_{i_{1}\cdots i_{m}}(x)\D x^{i_{1}}\cdots \D x^{i_{m}},
\] with $\{f_{i_{1}\cdots i_{m}}(x)\}$ symmetric in its indices, smooth and compactly supported. With repeating indices, Einstein summation convention will be assumed throughout this paper.
%Let $C_{0}^{\infty}(S^{m})$ %, \Dc'(S^{m})$ and  $\Ec'(S^{m})$ 
%denote the space of  compactly supported smooth tensor fields. %, tensor field-distributions and compactly supported tensor field distributions, respectively. 

Let $ T\Sb^2  = \{(\omega, x) \in  \Sb^2 \times \Rb^3: \omega\cdot x = 0\}$ be the tangent bundle of the unit sphere $\Sb^2 \subset \Rb^3$ and let
$$ T\Sb^2  \oplus T\Sb^2  = \{(\omega, x, y) \in \Sb^2 \times \Rb^3 \times \Rb^3: \omega \cdot x =0 ,  \omega \cdot y  =0\}$$
be the Whitney sum. 
\begin{definition}[\cite{Sharafutdinov_Book}]
	The transverse ray transform $\Tc: C_c^\infty(S^m) \rightarrow C^\infty(T\Sb^2  \oplus T\Sb^2)$ is the bounded linear map defined by 
	\begin{align*}
	\Tc f(\omega, x,y) =  \int_\Rb \langle f(x+t\omega), y^{\odot m}\rangle \D t,
	\end{align*}
	where $y^{\odot m}$ denotes the $m^{\mathrm{th}}$ symmetric tensor product of $y$ and $\langle f(x),y^{\odot m}\rangle$ is defined by $f_{i_{1}\cdots i_{m}}(x)y^{i_{1}}\cdots y^{i_{m}}$.
\end{definition}
%	Here $\langle f(x),y^{\odot m}\rangle$ is defined by $f_{i_{1}\cdots i_{m}}(x)y^{i_{1}}\cdots y^{i_{m}}$.
%\noindent 

We will find it more convenient to work with an equivalent vectorial version of TRT which we define below.
%We will work with the following equivalent form of the transverse ray transform. 
Let $\omega\in \Sb^{2}$ be represented in spherical coordinates by  
\[
\omega= \lb \cos \theta_1,\sin\theta_{1}\cos \theta_{2},\sin\theta_{1}\sin\theta_{2}\rb
\]
%\begin{equation}
%\begin{aligned}\label{spherical coordinate}
%\omega_1&= 	\cos\theta_1\\
%\omega_2&= 	\sin\theta_1\cos\theta_2\\
%\omega_3&= 	\sin\theta_1\sin\theta_2 \cos\theta_3\\
%\vdots  \\
%\omega_{n-1}&=\sin\theta_1\sin\theta_2 \cdots \sin\theta_{n-2}\cos\theta_{n-1}\\
%\omega_{n}&=\sin\theta_1\sin\theta_2 \cdots \sin\theta_{n-2}\sin\theta_{n-1}, 
%\end{aligned}
%\end{equation}
where $0 \leq \theta_1 <\pi$ and $0 \leq \theta_{2} < 2\pi$. Consider the orthonormal frame $\{\OM, \OM_1,\OM_2\}$ with $\OM_{1}$ and $\OM_2$ defined by 
\begin{align}\label{omegai's}
& \OM_{1}=\lb -\sin \theta_{1},\cos \theta_{1}\cos \theta_{2},\cos \theta_{1}\sin \theta_{2}\rb \mbox{ and }  \OM_{2}=\lb 0,-\sin \theta_{2},\cos \theta_{2}\rb.
\end{align}

%let $\omega_{1},\cdots ,\omega_{n-1}$ be any orthonormal frame of the vector space perpendicular to $\OM$. For instance, we can take 
%\Beq\label{omega i's}
%\OM_{i}=\frac{1}{\sin\theta_{1}\cdots \sin\theta_{i-1}}\OM_{\theta_{i}},
%\Eeq
%where $1\leq i\leq n-1$ and $\omega_{\theta_i} = \frac{\partial \omega}{\partial\theta_i}$.% for $ j= 1, \dots, (n-1)$ then it can be checked that $\omega, \omega_{\theta_1},\dots, \omega_{\theta_{n-1}}$
%%\par Using the above spherical coordinates, we give a parametrization $\Phi :\mathbb{R}^{n}  \times \mathbb{S}^{n-1}$to $T\mathbb{S}^{n-1}$ as follows:
We define the vectorial version of $\Tc$ as follows: 
\begin{definition}
	\label{Vectorial TRT}
	%	 Let $ f \in C^{\infty}_{0}(S^m)$  be a symmetric $m$-tensor field in $ \Rb^{n} $,  then 
	For $ 0 \leq i\leq m$, define $\Tc=\lb\mathcal{T}_{i}\rb
	: C^{\infty}_{c}(S^m) \to \lb C^\infty(T\Sb^{2})\rb^{m+1}$  by
	\begin{align}\label{def:TRT}
	\mathcal{T}_{i} f(x, \OM) = \int _{\mathbb{R}}  f_{j_1j_2 \cdots j_m}(x+t\omega)\omega^{j_1}_{1}\cdots \OM^{j_{m-i}}_{1} \omega^{j_{m-(i-1)}}_{2}\cdots \omega^{j_m}_{2} \D t.
	\end{align} 
\end{definition}
It is straightforward to see that these two definitions are equivalent.

%As already mentioned, we are interested in the inversion of $\Tc$, that is, to recover the symmetric $m$-tensor field $f$ from $\Tc f$. % It is well known %from the work of Sharafutdinov 
%\cite{Sharafutdinov_Book} that only the solenoidal component $f^{s}$ of a symmetric tensor field $f$ can be recovered from the transform $\Rc$. 
In $2$-dimensions, TRT and the standard ray transform \cite{Sharafutdinov_Book}, also called the longitudinal ray transform (LRT),  give equivalent information and it is well-known that the latter transform on symmetric tensor fields  has an infinite dimensional kernel. Hence it is not possible to reconstruct the full tensor field $f$ from its transverse ray transform in $2$-dimensions. Furthermore, the space of lines in $\Rb^{n}$ is $2n-2$ dimensional, and in dimensions $n\geq 3$, %Since the space of lines in $\Rb^{n}$ is $2(n-1)$ dimensional, 
the problem of recovery of $f$ from $\Tc f$ is over-determined. % for $n\geq 3$. 
Therefore a natural question is to investigate the inversion of $\Tc$ restricted to an $n$-dimensional data set.  We address this problem for the case of dimension $n=3$ in this paper, and the 3-dimensional set of lines we choose is the set of all lines passing through a fixed curve $\g\in \Rb^3$.

The inversion of TRT and of the corresponding non-linear problem appearing in polarization tomography has been considered in several prior works \cite{Sharafutdinov_Book,Novikov-Sharafutdinov,Sharafutdinov_PolarizationII,Holman, Holman2, Lionheart-Withers,Desai-Lionheart,Griesmaier2018,Krishnan-Monard-Mishra,Lionheart_2020}. With respect to the study of restricted TRT, we refer to the works \cite{Lionheart-Sharafutdinov,Desai-Lionheart}. Recently a support theorem for TRT in the setting of analytic simple Riemannian manifolds was considered by \cite{Abhishek-supporttheorem}.

We study the inversion of restricted TRT using  microlocal analysis techniques. We are interested in the reconstruction of singularities of the symmetric tensor field $f$ given its restricted TRT.  The study of generalized Radon transforms in the framework of Fourier integral operators began with the fundamental work of Guillemin \cite{Guillemin} and  Guillemin-Sternberg \cite{Guillemin-Sternberg-AJM}. Since then, microlocal analysis has become a very powerful tool in the study of tomography problems;  see \cite{GS, Greenleaf-Uhlmann-Duke1989,  Boman-Quinto-Duke, Boman1993,SU1, Katsevich2002, Lan2003,Ramaseshan2004, SU2, SU3,SSU,Uhlmann-Vasy,K1,Abhishek-Mishra}. Of these works, the  paper \cite{Greenleaf-Uhlmann-Duke1989} is a fundamental work where Greenleaf and Uhlmann studied a restricted ray transform on functions in the setting of Riemannian manifolds.  However, most of these works are done for LRT and to the best of our knowledge, other than the support theorem result \cite{Abhishek-supporttheorem}, we are not aware of any prior work that studies a restricted TRT from the view point of microlocal analysis.

Specifically, we study the microlocal inversion of the Euclidean TRT on symmetric $m$-tensor fields given the restricted data set consisting of all lines passing through a fixed curve $\g$ in $\Rb^{3}$. The transverse ray transform $\Tc$ defined in \eqref{def:TRT} restricted to lines passing through the curve $\g$ will be denoted by $\Tc_{\g}$ and its formal $L^{2}$ adjoint by $\Tc_{\g}^{*}$. We determine the extent to which the wavefront set of a symmetric $m$-tensor field can be recovered from the wavefront set of its restricted TRT. We are motivated by the related works done for restricted LRT \cite{Greenleaf-Uhlmann-Duke1989,LanThesis, Lan2003,Ramaseshan2004,Krishnan-Mishra} and we mainly follow the techniques from these works.

The article is organized as follows. \S \ref{Sect: Preliminary} is devoted to stating some preliminary results about the restricted TRT, to some fundamental results about distributions associated to two cleanly intersecting Lagrangians introduced in \cite{Melrose-Uhlmann, Guillemin-Uhlmann, Greenleaf-Uhlmann-Duke1989}, the microlocal results relevant for the analysis of our transform, and the statement of the main result.  %We do not give any proofs in this section as all the details follow in a straightforward manner from the works \cite{LanThesis, Lan2003, Krishnan-Mishra}.  
We give the proof of the main results in \S \ref{Sect:Principal Symbol} and \S \ref{Sect: Microlocal Inversion}. 
	\section{Preliminaries and statement of the main result}\label{Sect: Preliminary}
We first state precisely the conditions imposed on the curve $\g$, and the wavefront set directions that are potentially recoverable based on microlocal analysis of the restricted transverse ray transform.

Let $B$ be a ball in $\Rb^3$. Let $\g$ be a smooth regular curve without self-intersections in $\Rb^3$ defined on a bounded interval and with its range in the complement of $B$. We assume that there are uniform upper and lower bounds on the number of intersection points of almost every hyperplane passing through the set $B$ with the curve $\g$, and that the lower bound is at least $m+1$ (where $m$ is the order of the tensor field under consideration). This condition on the lower bound is a modified form of so-called Kirillov-Tuy condition. %, \textcolor{red}{ which is required to recover a symmetric $m$-tensor field.}
%also assume, there is a lower bound on the number of  almost all hyperplane $H$ in $\Rb^3$ intersecting $B$ also intersects $\g$ in at least $m+1$ distinct points $\g(t_1),\cdots,\g(t_{m+1})$, and there is a uniform upper bound on the number of intersection points of $H$ for almost all $x\in H$, any two of the vectors from $\{x-\g(t_1),\cdots,x-\g(t_{m+1})\}$ are linearly independent.

%For each hyperplane $H$ intersecting $B$, we consider the set of points $x\in B$ such that any two of the vectors from $\{x-\g(t_1),\cdots,x-\g(t_{m+1})\}$ are linearly independent. Together with the direction normal to the hyperplane $H$, as we vary the set of points can be viewed as elements of $T^{*}\setminus \{0\}$, which we denote by $\Xi$. 
For our microlocal analysis approach to work, we need to restrict ourselves to certain wavefront set directions that we can potentially recover. The sets defined below (see \cite{Greenleaf-Uhlmann-Duke1989, Ramaseshan2004,Krishnan-Mishra}) are motivated by this restriction. Given $(x,\xi)\in T^{*}B \setminus \{0\}$, we denote by $H(x,\xi)$, the plane passing through $x$ and perpendicular to $\xi$. The points $\{\g(t_i)\}$ on the curve $\g$ below refer to the points of intersection of the curve $\g$ and the hyperplane $H(x,\xi)$.

%We now denote the following sets to state precisely the recoverable singularities 

\begin{align}
\notag  \label{XiDef}\Xi =& \Big{\{}
(x,\xi)\in T^*B\setminus \{0\}:   \text{there exists at least $m+1$ points $\{ \g(t_j)\}_{j=1}^{m+1}$ } \text{such that  }\\
 	&\text{all pairs of vectors from }  \{\lb x-\g(t_{j})\rb \}_{j=1}^{m+1} \text{ are linearly independent}\Big{\}}.
\end{align}
\begin{align*}
\Xi'&= \Big{\{}
(x,\xi)\in \Xi :\text{all the intersection points of } H(x,\xi) \text{ with } \g \text{ are transverse}\Big{\}}.\\
\Xi'' &= \Big{\{}
(x,\xi)\in\Xi: \text{ the tangential intersection points $\{\g(t_j))\}$ satisfy} \langle \g^{\prime \prime}(t_{j}),\xi\rangle\neq 0 \Big{\}}.
\end{align*}
The potentially recoverable singularities belong to the union $\Xi'\cup \Xi''$ (see the statement of Theorem \ref{Main theorem} for a more precise description). Therefore, without loss of generality, in all the analysis below we will restrict ourselves to the cotangent directions in this union. Below we give an example of a curve $\g$ satisfying the Kirillov-Tuy condition for vector fields and also discuss the corresponding $\Xi$, $\Xi'$, and  $\Xi''$.
\begin{example}[\cite{Tuy,Vertgeim}]\label{Example 2.1}
	For $ m=1 $ (vector fields), consider the curve $ \g $ as the union of three orthogonal great circles on the sphere of radius 2 (the equator and the meridians of $ 0^{\circ} $ and $ 90^{\circ} $) and center at the origin. Then every plane $H$ intersecting  the unit ball $B$ will intersect the curve $\g$ at least two different points $\g_1$ and $\g_2$. And for almost every $x \in H$, the vectors $x - \g_1$ and $x - \g_2$ are linearly independent.
\end{example}
With respect to above example, we have $ \Xi = \Xi'=T^{*}B\setminus\{0\} = B\times \mathbb{R}^{n}\setminus\{0\}$ and $\Xi''$ is the empty set. In this case, the potentially recoverable singularities consists of the set $ \Xi' $. Instead of three great circles, if we consider the curve $ \g $  to be union of two orthogonal great circles of radius $ 2 $, then the curve $ \g $ satisfies the Kirillov-Tuy condition of order one for almost every hyperplane. In that case, we have to exclude a set of measure zero from $ \Xi=\Xi'=T^{*}B\setminus\{0\} = B\times \mathbb{R}^{n}\setminus\{0\}$ and $ \Xi'' $ is the same as above.

%\subsection{Conditions on the curve $\gamma$}
%	\begin{enumerate}
%	    \item The curve $\g$ defined on a bounded interval is smooth, regular
%	    \item There is a uniform upper and lower bounds on the number of intersection points of $\gamma$ with almost every hyperplane, and the lower bound is $m+1$
%	    \item 
%	\end{enumerate}

Next we state some preliminary microlocal results concerning the operators $\Tc_{\g}$ and $\Tc_{\g}^{*}\Tc_{\g}$. The proofs of these statements follow by suitable adaptations of the arguments given in \cite{LanThesis,Krishnan-Mishra} and therefore we skip them.

%We assume that the curve $\g$ is defined on a bounded interval, is smooth and there is a uniform bound on the number of intersection points of $\g$ with almost every plane in $\Rb^3$ \cite{LanThesis,Krishnan-Mishra}. We further assume that $\g$ satisfies the Kirillov-Tuy condition.
%\begin{enumerate}
%	\item The curve $\g: I\to \Rb^{3}$, where $I$ is a bounded interval, is smooth, regular and without self-intersections.
%	\item There is a uniform bound on the number of intersection points of almost every plane in $\Rb^{3}$ with the curve $\g$, see \cite{LanThesis}.
%\item The curve $\g$ satisfies the Kirillov-Tuy %condition; see Definition \ref{K-T Defn} below.
%\end{enumerate}
%	\begin{definition}[Kirillov-Tuy condition, \cite{Krishnan-Mishra}]\label{K-T Defn} Consider a ball $B$ in $\Rb^{3}$. We say that a smooth curve $\g$ defined on a bounded interval satisfies the Kirillov-Tuy condition of order $m \geq 1$ if for almost all planes $H$ in $\Rb^{3}$ intersecting the ball $B$, there are at least $(m+1)$ points $\g(t_1),\dots ,\g(t_{m+1})$ in $H\cap \g$.

Let us denote by $\Cc$, the line complex consisting of all lines  passing through the curve $\g$. 
Let $\ell$ be a line in $\Cc$ and 
\[
Z=\{(\ell,x): x\in \ell\} \subset \Cc\times \Rb^{3}\]
be the point-line relation. For given $t$ (in the domain of $\g$) and $\omega = (\theta_1, \theta_2) \in \Sb^2$, we can define a unique line $\ell \in \Cc$ by $\ell =\{\g(t) + s \omega : s \in \Rb\}$. Therefore,  we have that $(t,\OM,s)$ is a local parametrization of $Z$. The conormal bundle of $Z$ is given by
 $$N^{*}Z = \left\{ \left(\ell, x; \Gamma, \xi\right) : (\ell, x) \in Z \text{ and } (\Gamma, \xi)|_{T_{(\ell, x)}Z}  = 0 \right\}.$$
It has been shown in \cite{LanThesis,Krishnan-Mishra} that $N^{*}Z$ can be parametrized by $\{(t,\OM,s,\Gamma,\xi)\}$ where 

\Beq\label{The vector xi}  
\xi = z_{1}\OM_{1} + z_{2}\OM_{2}\mbox{ for some } z_{1} \mbox{ and } z_{2}\in \Rb,
\Eeq 
and $\OM_{1},\OM_2$ are given by \eqref{omegai's}, and 
\Beq\label{Gamma}
\Gamma=
\begin{pmatrix}
	\Gamma_{1}\\
	\Gamma_{2}\\
	\Gamma_{3}
\end{pmatrix}=\begin{pmatrix}
	-\xi \cdot \g'(t)\\ -s z_{1}\\ -s z_{2} \sin \theta_{1}\\
\end{pmatrix}.
\Eeq

\begin{lemma} The  map 
	\[
	\Phi: (t, \theta_{1},\theta_{2},s,z_{1},z_{2})\to (t, \theta_{1},\theta_{2},\Gamma; x,\xi)
	\]
	with $\Gamma$ as in \eqref{Gamma}, $\xi$ as in \eqref{The vector xi} and $x=\g(t)+s\OM$ gives a local parametrization of $N^{*}Z$ at the points where $\theta_{1}\neq 0,\pi$.
\end{lemma}
\begin{proposition}
	Each component of the operator $\Tc_{\g}$ is a Fourier integral operator of order $-1/2$ with the associated canonical relation $C$ given by $(N^{*}Z)'$ where $Z=\{(\ell,x): x\in \ell\}$.
	The left and the right projections $\pi_{L}$ and $\pi_{R}$ from $C$ drop rank simply by $1$ on the  set 
	\Beq\label{The set Sigma}\Sigma:=
	\{(t,\theta_{1},\theta_{2}, s, z_{1},z_{2}): \g'(t)\cdot \xi=0\},
	\Eeq where $\xi$ is given by \eqref{The vector xi}. The left projection $\pi_{L}$ has a blowdown singularity along $\Sigma$ and the right projection $\pi_{R}$ has a fold singularity along $\Sigma$. 
\end{proposition}
We refer the reader to \cite{GG} for the definitions of fold and blowdown singularities.
\begin{lemma}
	The wavefront set of the Schwartz kernel of $\Tc_{\g}^{*}\Tc_{\g}$ satisfies 
	\[
	WF(\Tc_{\g}^{*}\Tc_{\g})\subset \Delta \cup \Lambda,
	\]
	where $\Delta$ and $\Lambda$ are defined 
	as follows: %is the diagonal Lagrangian 
	\Beq\label{Diagonal relation}
	\Delta =\left\{(x,\xi; x,\xi): x= \g(t) + s\theta, \xi\in \theta^{\perp}\setminus \{0\}\right\} \mbox{ and }
	\Eeq
	\begin{align}\label{Artifact relation}
	\Lambda =\left\{\left(x,\xi,y,\frac{\tau}{\tilde \tau}\xi\right):x=\g(t)+\tau\theta,y=\g(t)+\tilde\tau\theta,\xi \in \theta^\perp\setminus \{0\}, \g'(t)\cdot\xi=0, \tau\neq 0\neq \tilde{\tau} \right\}.
	\end{align}
	The condition imposed on the curve in the definition of $ \Xi'' $ entails the clean intersection of the sets $\Delta$ and $\Lambda$. We have 
	\[
	\Delta \cap \Lambda=\{(x,\xi; x,\xi): x=\g(t)+ s\theta, \xi\in \theta^{\perp}\setminus\{0\}, \g'(t)\cdot \xi=0\}. 
	\]
$\Delta \cap \Lambda$ is a smooth submanifold of codimension $k=1$ in both $\Delta$ or $\Lambda$.
	
	% due to the third condition in \eqref{The three sets}.
	%in \eqref{Diagonal relation} and \eqref{Artifact relation}, respectively.
	
	% is the diagonal Lagrangian 
	%\[
	%\Delta =\left\{(x,\xi; x,\xi): x= \g(t) + s\theta, \xi\in \theta^{\perp}\setminus \{0\}\right\}
	%\]
	%and 
	%\begin{align*}
	%	\Lambda =\left\{(x,\xi,y,\frac{\tau}{\tilde \tau}\xi):x=\g(t)+\tau\theta,y=\g(t)+\tilde\tau\theta,\xi \in \theta^\perp\setminus \{0\}, \g'(t)\cdot\xi=0, \tau\neq 0,\tau' \neq 0 \right\}.
	%\end{align*}
\end{lemma}
%\newpage
\begin{lemma} \cite{LanThesis} The Lagrangian $\Lambda$ defined in \eqref{Artifact relation} arises as a flowout from the set $\pi_{R}(\Sigma)$.
\end{lemma}
%\section{Principal symbol of the operator $\Tc_{\g}^{*}\Tc_{\g}$}\label{Sect:Principal Symbol}
\subsection{Paired Lagrangian distributions}
We will analyze the operators $\Tc_{\g}$ and $\Tc_{\g}^{*}\Tc_{\g}$ in the framework of $I^{p,l}$ classes of distributions. We refer the reader to the three seminal works on this subject \cite{Melrose-Uhlmann, Guillemin-Uhlmann, Greenleaf-Uhlmann-Duke1989}.
For the convenience of the reader, we give a quick summary of the properties of  the $I^{p,l}$ class of distributions \cite{Guillemin-Uhlmann} that we require in this paper. 

Let $u\in I^{p,l}(\Delta,\Lambda)$, where $\Delta$ and $\Lambda$ are two cleanly intersecting Lagrangians with intersection $\Sigma=\Delta \cap \Lambda$. As an example, the reader may take $\Delta$ and $\Lambda$ from \eqref{Diagonal relation} and \eqref{Artifact relation}. 

Then 
\begin{enumerate}
	\item  $ WF(u) \subset \Delta\cup \Lambda$.
	\item Microlocally, the Schwartz kernel of $u$ equals the Schwartz kernel of a pseudodifferential operator of
	order $p + l$ on $\Delta\setminus\Lambda$ and that of a classical Fourier integral operator of order $p$
	on $\Lambda\setminus\Delta$.
	\item $I^{p,l}\subset I^{p^\prime,l^\prime} \text{ if } p\leq p^\prime \text{ and } l\leq l^\prime$.
	\item  $\cap_lI^{p,l}(\Delta,\Lambda) \subset I^p(\Lambda)$.
	\item $\cap_p I^{p,l}(\Delta,\Lambda) \subset $ The class of smoothing operators.
	\item The principal symbol $\sigma_{0}(u)$ on $\Delta \setminus \Sigma$ has the singularity on $\Sigma$ as a conormal distribution of order $l-\frac{k}{2}$, where $k$ is the codimension of $\Sigma$ as a submanifold of $\Delta$ or $\Lambda$.
	\item  If the principal symbol  $\sigma_{0}(u) = 0$ on $ \Delta \setminus \Sigma$, then $ u \in I^{p,l-1}(\Delta,\Lambda)+I^{p-1,l}(\Delta,\Lambda)$.
	\item $u$ is said to be elliptic if the principal symbol $\sigma_{0}(u)\neq 0$ on $\Delta\setminus \Sigma$ if $k\geq 2$, and for $k=1$, if $\sigma_{0}(u)\neq 0$ on each connected component of $\Delta\setminus \Sigma$.

\end{enumerate}

The Lagrangian $\Lambda$ defined in \eqref{Artifact relation} arises as a flowout, and the main tool in the construction of a relative left parametrix for our operator $\Tc_{\g}^{*}\Tc_{\g}$ is the following composition calculus due to  Antoniano and Uhlmann \cite{AntonianoandUhlmann}: 
\begin{theorem}[\cite{AntonianoandUhlmann}]\label{Th2.1}
	If $A \in I^{ p,l} (\Delta,\Lambda )$ and $B \in I^{ p^\prime,l^\prime} ( \Delta,\Lambda)$, then composition of $A$ and $B$,
	$A\circ B \in I^{ p+p^\prime +\frac{k}{2},l+l^\prime - \frac{k}{2}} (\Delta,\Lambda)$ and the prinicipal symbol, $\sigma_0 (A \circ B) = \sigma_0 (A)\sigma_0 (B)$, where, $k$ is the codimension of $\Sigma$ as a submanifold of either $\Delta$ or $\Lambda$.
\end{theorem}

Let $B$ be the ball that appears in the definition of the set $\Xi$ (see \eqref{XiDef}) above.
%We define the sets $\Xi_{\Delta}$ and $\Xi_{\Lambda}$ below which are motivated by the sets 
%Denote the plane passing through $x$ and perpendicular to $\xi$ by $x+\xi^{\perp}$.
Let $K\subset \Xi'$ be a closed conic subset. The space of compactly supported distributions in $B$ whose wavefront set is contained in $K$ will be denoted by $\mathcal{E}^\prime_K(B)$. We now state the main result.

\begin{theorem}\label{Main theorem} Let  
	$\Xi_0 \subseteq \Xi'$ be such that $\overline{\Xi}_0 \subseteq \Xi' \cup \Xi''$ and $K$ be a closed conic subset of \  $\Xi_0$. 
	%Let $\mathcal{E}^\prime_K(B)\subset \Ec'(B)$ denote the space of compactly supported distributions in $B$ whose wavefront set is contained in $K$. 
	There exists an operator $\mathcal{B}\in I^{0,1}(\Delta,\Lambda)$ and an operator $\Ac\in I^{-1/2}(\Lambda)$ such that for any symmetric $m$-tensor field $f$ with coordinates in $\mathcal{E}^\prime_K(B)$, 
	\[
	\mathcal{B}\mathcal{T}_{\g}^{*}\mathcal{T}_{\g}f  = f + \mathcal{A}f + \text{smoothing terms.}
	\]
	%Here $\Delta$ is the diagonal relation of $T^*\mathbb{R}^n\times T^*\mathbb{R}^n$ and $\Lambda$ arises as a flowout of a Hamiltonian vector field.
\end{theorem}
\begin{remark}
The condition $ \Xi_0\subseteq \Xi' $ implies that the points $ (x,\xi,x,\xi)\in \Delta \setminus \Sigma $ and also gives the ellipticity of the operator $ \Tc_{\g}^{*}\Tc_{\g} $; see \eqref{principal symbol}.  Furthermore, the condition $ \overline{\Xi}_0 \subseteq \Xi' \cup \Xi'' $ ensures the applicability of the functional calculus from \cite{AntonianoandUhlmann} (see the statement of Theorem \ref{Th2.1} above).
\end{remark}
%But the curve $ \g $ intersects  any plane at finitely many points and contribution from all non-degenerate critical points appears as sum in the symbol of $\Tc_{\g}^{*}\Tc_{\g}  $ [see  \eqref{principal symbol}]. 
The proof of this result is based on a suitable adaptation of the techniques from  \cite{Greenleaf-Uhlmann-Duke1989,Lan2003,Ramaseshan2004,Krishnan-Mishra} to the TRT setting. To this end, we compute the principal symbol of the operator $\Tc_{\g}^{*}\Tc_{\g}$ on the diagonal $\Delta$ away from the set $\Sigma$ and use this to construct a relative left parametrix for this operator. 
Our inversion procedure introduces an additional error term (in addition to smoothing terms) because we are working with a restricted transverse ray transform. This error term is a Fourier integral operator associated to the known Lagrangian $\Lambda$; see \eqref{Artifact relation}.
%	Since we deal with a restricted transverse ray transform, the inversion procedure  introduces an additional error term (in addition to smoothing terms), but this error term is a Fourier integral operator associated to the known Lagrangian $\Lambda$; see \eqref{Artifact relation}.
\section{Principal symbol of the operator $\mathcal{T}_\gamma^{*}\mathcal{T}_\gamma$}\label{Sect:Principal Symbol}

In this section, we give the principal symbol matrix of the operator $\Tc_{\gamma}^{*}\Tc_{\gamma}$ and show that it is elliptic. 

The operator $\Tc_{\gamma}^{*}\Tc_{\gamma}$ can be written as  
% the pre give the principal symbol of the Schwartz kernel of the operator 
\[
\Tc_{\g}^{*}\Tc_{\g}= \sum\limits_{i=0}^{m}\left[ \Rc_{\g}^*\left( \OM^{j_1}_{1} \cdots \OM^{j_{m-i}}_{1}\OM^{j_{m-(i-1)}}_{2}\cdots \omega^{j_m}_{2} \omega^{l_1}_{1} \cdots \OM^{l_{m-i}}_{1}\omega^{l_{m-(i-1)}}_{2}\cdots \omega^{l_m}_{2}\right)\Rc_{\g}\right],\]
where $\Rc_{\g}$ is the restricted scalar ray transform (that is, ray transform of functions) and $\Rc_{\g}^{*}$ is its formal $L^{2}$ adjoint. The set $K$ below is as in the statement of Theorem \ref{Main theorem}.

\begin{proposition}\label{prop:principal symbol}
	The principal symbol matrix $A_{0}(x,\xi)$ of the operator $\Tc_{\g}^{*}\Tc_{\g}$ for $(x,\xi) \in K$ is 
	{\footnotesize
		\begin{align}\label{principal symbol}
		A_{0}(x,\xi)=\sum\limits_{j}\sum\limits_{i=0}^{m}\frac{2\pi\omega^{j_1}_{1}(t_j)\cdots \omega^{j_{m-i}}_{1}(t_j)\OM^{j_{m-(i-1)}}_{2}\cdots \omega^{j_m}_{2}(t_j) \omega^{l_1}_{1}(t_j)\cdots\OM^{l_{m-i}}_{1}(t_j)\omega^{l_{m-(i-1)}}_{2}(t_j)\cdots \omega^{l_m}_{2}(t_j)}{|\xi| \lvert (\g'(t_{j}(\xi_{0}))\cdot \xi_0)\rvert \lvert(\g(t_j(\xi_{0}))-x)\rvert}.
		\end{align}
	}
	In \eqref{principal symbol} above, $\xi_0$ is the unit vector in the direction of $\xi$, $j$ varies over the number of  intersection points of the plane $H(x,\xi)$ with the given curve $\g$. 
\end{proposition}
The derivation of this formula is similar to the one in \cite{LanThesis,Ramaseshan2004,Krishnan-Mishra} and therefore we do not give the details here. % skip it.

\begin{proposition}\label{ellipticity}
For $(x, \xi) \in K$, the principal symbol matrix $A_{0}(x,\xi)$ for $\xi\neq 0$ is injective. 
\end{proposition}
\begin{proof}
	For $(x,\xi)\in T^* \mathbb{R}^3 \setminus {\{0\}} $, % we denote  $ x+\xi^{\perp} $  the hyperplane passing through $ x $ perpendicular to $ \xi $.  With the plane $x+\xi^{\perp}$ fixed,
	without loss of generality, we choose a spherical coordinate system such that $\OM(\cdot)$ and $\OM_{1}(\cdot)$ are parallel to the plane $H(x,\xi)$ and $\OM_{2}(\cdot)$ is in the direction of $\xi$. 
	
	The plane $H(x,\xi)$ intersects the curve $\g$ in at least $ (m+1) $ points, say $t_1,\cdots,t_{m+1},\cdots, t_{j^\prime}$. %and any two of the vectors from $\{x-\g(t_{k}) : 1\leq k \leq l\}$ are linearly independent by the definition of $\Xi$. 
	
	Denote the collection of unit vectors in the directions $x-\g(t_{1}), \cdots, x-\g(t_{j^\prime})$ by
	%from $ x $ to $\lb H(x,\xi) \cap \g\rb$ as
	\begin{align*}
	\mathbb{A}= \bigg\{\OM(t_{j})=\frac{x-\g_{j}}{|x-\g_{j}|}: \g_j=\g(t_j), 1\leq j\le j'\bigg  \} \mbox{ where } j'\geq m+1.
	\end{align*} 
	Now any two of the vectors in $\mathbb{A}$ are linearly independent since $(x,\xi)\in \Xi$. This in turn implies that for almost all points $x$, any two of the vectors in the collection 
	\begin{align*}
	\mathbb{A}'= \bigg\{\OM_{1}(t_{j}):  1\leq j\le j' \bigg  \}
	\end{align*} 
	where, recall, $\OM_{1}(t_j)$ corresponding to $\OM(t_j)$ defined in \eqref{omegai's}, is also linearly independent.	
	
	Denote the matrix $U_{q}= U_{\underbrace{1\cdots 1}_{q}\underbrace{2 \cdots 2}_{m-q}}$, for $0\le q\le m$, whose columns are
	\begin{align*}
	\lb \frac{2\pi}{|\xi| \lvert (\g'(t_{j}(\xi_{0}))\cdot \xi_0)\rvert \lvert(\g(t_j(\xi_{0}))-x)\rvert}\rb^{1/2}\omega_{1}(t_j)^{\odot q} \odot \omega_{2}(t_j)^{\odot m-q}\mbox{ for } 1\leq j \leq j',
	\end{align*}
where $\odot$ denotes the symmetric tensor product.
	Let us denote the matrix $P$ with column blocks $\{U_{i}\}, \, 0\leq i\leq m$: 
	\begin{align}\label{Matrix P}
	P = \begin{pmatrix}
	U_{m} & U_{m-1}& \cdots& U_{q} & \cdots&U_{0}    
	\end{pmatrix}
	\end{align}
	Note that the number of rows in $P$ is $(m+2)(m+1)/2$.
	
	We have
	\begin{align*}
	A_{0}(x,\xi)= PP^t,			\end{align*}
	with $P$ defined in \eqref{Matrix P}. In Lemma \ref{adding rank}, we show that Rank$(P)=(m+2)(m+1)/2$.
	Since $P$ has real entries, $\mbox{Rank}(PP^{t})= \mbox{Rank}(P) $. Therefore the principal symbol matrix $ A_{0}(x,\xi) $ has full rank on $ \Delta \setminus \Sigma $. %This implies $ A_{0}(x,\xi)f =0 $ implies $f=0 $.
	\end{proof}
\begin{lemma}\label{Lem3.3}
	For $ q \geq 1$, consider a collection of $q+1$ pair-wise independent vectors $v_1,\cdots,v_{q+1}$ in $\Rb^3$. Then the collection of vectors
	\[
	v_1^{\odot q},\cdots,v_{q+1}^{\odot q}
	\]
	is also linearly independent.
\end{lemma}
\begin{proof}
	%We prove the result by induction on $p$.
	We can write 
	$v_{i}=c_{i1} v_1+ c_{i2} v_2$ for $i\ge 3$ and for two non-zero constants $c_{i1}$ and $c_{i2}$. %We then write 
	%\[
	%v(t_{k})^{\odot p+1}=v(t_{k})^{\odot p}\odot \lb c_{k1} v(t_1)+ c_{k2} v(t_2)\rb \mbox{ for } k\geq 3.\]
	
	Assume 
	\[
	\sum\limits_{i=1}^{q+1} d_{i} v_i^{\odot q}=0,
	\]
	for some nonzero constants $d_{i}$.
	Then using the above, we have 
	\[
	d_1 v_1^{\odot q} +d_2v_2^{\odot q}+\sum\limits_{i=3}^{q+1} d_{i} \lb c_{i1}v_1+c_{i2}v_2 \rb^{\odot q}=0
	\]
	From this, we get,
	\[
	\lb d_1 + \sum\limits_{i=3}^{q+1}d_ic^{q}_{i1} \rb v_1^{\odot q}+ \sum\limits_{i=3}^{q+1}\lb \sum\limits_{r=1}^{q-1} \binom{q}{r} d_{i} c^{q-r}_{i1} c^{r}_{i2}\rb v_1^{\odot q-r}\odot v_2^{\odot r} + \lb d_2 + \sum\limits_{i=3}^{q+1}d_ic^{q}_{i2}\rb v_2^{\odot q}=0.
	\]
	Since $v_1$ and $v_2$ are linearly independent, the collection of tensors $\{v_1^{\odot q-r} \odot v_2^{\odot r}: 1\le r\le q-1\}$ is also linearly independent. Thus  
	\Beq\label{Eq1}
	\lb d_1 + \sum\limits_{i=3}^{q+1}d_ic^{q}_{i1} \rb=0,
	\Eeq
	\Beq\label{Eq2}
	\lb d_2 + \sum\limits_{i=3}^{q+1}d_ic^{q}_{i2} \rb=0
	\Eeq
	and
	\Beq\label{Eq3} \lb \sum\limits_{i=3}^{q+1} \binom{q}{r} d_{i} c^{q-r}_{i1} c^{r}_{i2}\rb=0, \quad \mbox{for}\quad 1\le r\le q-1.
	\Eeq
	Since $c_{i1}$ and $c_{i2}$ are both non-zero constants for all $i$, by factoring out $c_{i1}c_{i2}$, the system of equations in \eqref{Eq3} can be written as 
	\[\lb \sum\limits_{i=3}^{q+1}  d_{i} c^{q-r-1}_{i1} c^{r-1}_{i2}\rb=0, \quad \mbox{for}\quad 1\le r\le q-1.  \]
	This can written as a matrix system 
	\begin{align*}
	BY=0,
	\end{align*}
	where 
	\begin{align}\label{matrix system}
	B= \begin{pmatrix}
	c_{31}^{q-2} &  c_{41}^{q-2} &\cdots &c_{q+1,1}^{q-2}\\
	c_{31}^{q-3}c_{32} &c_{41}^{q-3}c_{42}&\cdots & c_{q+1,1}^{q-3}c_{q+1,2}\\
	\vdots & \vdots & \ddots&\vdots\\
	c_{32}^{q-2} & c_{4,2}^{q-2}& \cdots & c_{q+1,2}^{q-2}
	\end{pmatrix}
	\end{align}
	and $Y=(d_3,\cdots,d_{q+1})^t$. 
	
	Let 
	\[
	q_{i}= \frac{c_{i2}}{c_{i1}} \mbox{ and } b_{i} = (c_{i1},c_{i2}) \mbox{ for } 3\le i\le q+1.
	\]
	Since any two vectors from $\{ v_{i} : 3\le i\le q+1 \}$
	are linearly independent, we have that any two vectors from the set $\{ b_{i}: 3\le i \le q+1\} $ are also linearly independent. This gives that the ratios $q_{i}$'s are all distinct. 
	
	We are interested in proving that $\mbox{Kernel}(B)=\{0\}$. It is enough to prove that $\mbox{Kernel}(B^t)= \{0\}$. Now  \[B^tX=0,\] gives
	\[ \sum\limits_{r=0}^{q-2} c^{q-2-r}_{i1}c_{i2}^{r} e_{r}=0, \]
	for $3\le i \le q+1$ and $X=(e_0,\cdots,e_{q-2})^t$.
	%Dividing the above the system equations by a non-zero constant $ c^{p-2}_{i1}$ we obtain
	This in turn gives
	\begin{equation}\label{roots of polynomial}
	\sum\limits_{r=0}^{q-2} q_{i}^{r} e_{r}=0, \quad \mbox{for} \quad 3\le i \le q+1.
	\end{equation}
	We arrive at a Vandermonde matrix and hence $X=0$. This then implies that $Y=0$. Now going back to \eqref{Eq1} and \eqref{Eq2}, we have that $d_1=d_2=0$.
	%By the induction assumption, $d_i=0$ for all $1\leq i\leq p+2$, since $c_{i1}$ and $c_{i2}$ are non-zero for any $i\geq 3$.
\end{proof}
\begin{lemma}\label{rank lemma} The matrix $U_q$ satisfies 
	$\mbox{Rank}(U_q)\geq q+1$.
\end{lemma}
\begin{proof}
	We are interested in computing the principal symbol at points in $\Xi$. We have at least $m+1$ pairwise linearly independent vectors $\OM(t_{1}),\cdots, \OM(t_{m+1})$. The corresponding perpendicular vectors $\OM_{1}(t_{1}),\cdots, \OM_{1}(t_{m+1})$ are pairwise linearly independent and are also perpendicular to $\xi$. Now the collection of vectors $\{ \OM_{1}(t_{1})^{\odot q},\cdots, \OM_{1}(t_{q+1})^{\odot q}\}$ has rank $q+1$ by Lemma \ref{Lem3.3}. Therefore the rank of the matrix whose columns are $\OM_{1}(t_{1})^{\odot q},\cdots,\OM_{1}(t_{m+1})^{\odot q}$ is at least $q+1$. Finally, the rank of $U_{q}$ is at least $q+1$ as well, since $\OM_{2}(t_{k})$'s are in the direction of the nonzero vector  $\xi$.
	\end{proof}
	%We will be able to conclude that the rank is exactly $p+1$ as a consequence of the next lemma.

\begin{lemma}\label{adding rank1}
	Consider an arbitrary $U_{s}$ for $0\leq s\leq m$. Assume that the  values of $t_{k}$ corresponding to $s+1$ linearly independent columns of $U_{s}$ are $t_{j_{1}},\cdots, t_{j_{s+1}}$. Any column among these $s+1$ linearly independent columns cannot be written as a linear combination of the columns of the matrices $U_q$ for $ 0\le q\le m, q\neq s $ and the remaining  $s$ linearly independent columns of the matrix $ U_s $.
\end{lemma}
\begin{proof}
	After reordering, we may assume with loss of generality that $t_{j_{i}}=t_{i}$ for $1\leq i\leq s+1$.		Fix one of the linearly independent columns from $U_s$, say,  $ \OM_1(t_{1})^{\odot s}\odot \xi^{\odot m-s}$ (note that since $\OM(t_1)$ and $\OM_1(t_1)$ are parallel to the plane $H(x,\xi)$, $\OM_2(t_1)$ is in the direction of $\xi$). Suppose there exists constants $c_{qi}$'s and $ d_{j} $'s such that
	\begin{align}\label{relation}
	\OM_{1}(t_{1})^{\odot s}\odot \xi^{\odot m-s}= \sum_{q=0,q\neq s}^{m}\sum_{i=1}^{j'}c_{qi} \omega_{1}(t_i)^{\odot q}\odot \xi^{\odot m-q} +\sum_{j=2}^{s+1}d_{j} 	\omega_{1}(t_{j})^{\odot s}\odot \xi^{\odot m-s}.	
	\end{align}
	We write $\omega_{1}(t_{i})= \sum_{j=1}^{2} a_{ij}\omega_{1}(t_{j}) $
	for $ i\ge 3 $ for some constants $a_{ij}$. Substituting this above, we have,
	{\footnotesize
		\begin{align}
		\label{relation1}	\omega_{1}(t_{1})^{\odot s}\odot \xi^{\odot m-s}=
		& \sum_{q=0,q\neq s}^{m} \left(c_{q1}\omega_{1}(t_{1})^{\odot q}+c_{q2} \omega_1(t_{2})^{\odot q}+ \sum_{i=3}^{j'}c_{qi} \left(\sum_{j=1}^{2} a_{ij}\omega_{1}(t_{j})\right)^{\odot q} \right)\odot \xi^{\odot m-q}\\& 
		\label{relation2}+  \sum_{j=3}^{s+1}d_{j}\lb a_{j1}\rb ^s \omega_{1}(t_{1})^{\odot s}\odot \xi^{\odot m-s} + \left(d_{2}+\sum_{j=3}^{s+1}{d}_{j}\lb a_{j2}\rb ^s\right)\omega_1(t_{2})^{\odot s}\odot \xi^{\odot m-s}\\&
		\label{relation3}+ \sum_{r=1}^{s-1}\sum_{j=3}^{s+1}\wt{d}_{j} \lb a_{j1}\rb ^{s-r} \lb a_{j2}\rb ^{r}\omega_{1}(t_{1})^{\odot s-r}\odot \omega_{1}(t_{2})^{\odot r}\odot\xi^{\odot m-s}
		\end{align}
	}
	In the sum above, \eqref{relation1} is the expansion of the first summand in \eqref{relation} in terms of $\OM_1(t_1)$ and $\OM_1(t_2)$, \eqref{relation2} contains terms involving the powers of $\OM_1(t_1)^{\odot{s}}$ and $\OM_1(t_2)^{\odot{s}}$ when the second summand in \eqref{relation} is expanded in terms of $\OM_1(t_1)$ and $\OM_1(t_2)$, and \eqref{relation3} consists of the remaining terms from the second summand in \eqref{relation}. Also $\wt{d}_{j}$ are certain new constants involving $d_{j}$'s and binomial coefficients.
	This implies, for certain constants $c_{r_{1}r_{2}}$, 	
	\begin{align*}
	&\sum_{q=0,q\neq s}^{m}\sum_{r_1+r_2=q} c_{r_1r_2} \omega_{1}(t_{1})^{\odot r_1}\odot \omega_{1}(t_{2})^{\odot r_2}\odot \xi^{\odot m-q}\\&+ \left( \sum_{j=3}^{s+1}d_{j}\lb a_{j1}\rb ^s-1 \right)\omega_{1}(t_{1})^{\odot s}\odot \xi^{\odot m-s} + \left(d_{2}+\sum_{j=3}^{s+1} {d}_{j}\lb a_{j2}\rb ^s\right)\omega_{1}(t_{2})^{\odot s}\odot \xi^{\odot m-s}\\&+ \sum_{r=1}^{s-1}\sum_{j=3}^{s+1}\wt{d}_{j} \lb a_{j1}\rb ^{s-r} \lb a_{j2}\rb ^{r} \omega_{1}(t_1)^{\odot s-r}\odot \omega_{1}(t_2)^{\odot r}\odot\xi^{\odot m-s}=0.
	\end{align*}
	The vectors $ \{ \omega_{1}(t_{1}),\omega_{1}(t_{2}),\xi\} $ are linearly independent. Therefore the collection of tensors $\{\omega_{1}(t_{1})^{\odot j_1}\odot\omega_{1}(t_{2})^{\odot j_2}\odot\xi^{\odot j_3}: j_1+j_2+j_3= m\}$ is also linearly independent.  
	Thus 
	\begin{align}
	&\label{6.31}c_{r_1r_2}=0\\
	&\label{6.32}\sum_{j=3}^{s+1}{d}_{j}\lb a_{j1}\rb ^s-1=0\\
	&\label{6.33}d_{2}+\sum_{j=3}^{s+1}\wt{d}_{j}\lb a_{j2}\rb ^s=0\\
	&\label{6.34}\sum_{j=3}^{s+1}\wt{d}_{j} \lb a_{j1}\rb ^{s-r} \lb a_{j2}\rb^{r}=0 \ \ \ \mbox{for}\  1\le r\le s-1.
	\end{align}
	Note that the product $a_{j1}a_{j2}$ appears as a factor in \eqref{6.34} and since $a_{j1} $ and $ a_{j2} $ are both non-zero, we can cancel it out and after this write \eqref{6.34} as a matrix system:  
	\begin{align*}
	AX=0,
	\end{align*}
	where	
	\begin{equation}
	A=\begin{pmatrix}
	a_{31}^{s-2} &  a_{41}^{s-2} &\cdots &a_{s+1,1}^{s-2}\\
	a_{31}^{s-3}a_{32} &a_{41}^{s-3}a_{42}&\cdots & a_{s+1,1}^{s-3}a_{s+1,2}\\
	\vdots & \vdots & \ddots&\vdots\\
	a_{32}^{s-2} & a_{4,2}^{s-2}& \cdots & a_{s+1,2}^{s-2}
	\end{pmatrix}
	\end{equation}
	and 
	\begin{equation*} 
	X=(\wt{d}_3,\wt{d}_4,\cdots,\wt{d}_{s+1})^{t}.
	\end{equation*}
	%	Let $ b_{j}=(a_{j1},a_{j2}) $ for $ 3\le j \le s+1 $. Since we are only interested in computing the principal symbol at points in $\Xi$, we have that any two vectors from $\{ \omega_{1}(t_{j}) : 3\le j\le s+1 \}$
	%	are linearly independent. Hence any two vectors from the set $\{ b_{j}: 3\le j \le s+1\} $ are also linearly independent and therefore the columns of $A$ are linearly independent. %columns of $ A $ are $\{b_{j}^{\odot l-2}, 3\le j \le l+1\}$. 
	%Therefore by the Kirillov-Tuy condition for $m=l-2$, we have that the matrix $ A $  has full rank. 
	Now the argument proceeds exactly as in Lemma \ref{Lem3.3}.	Therefore we have $\{\wt{d}_{j}=0, 3\le j\le s+1\}$, and this implies $d_{j}=0$ for $3\leq j\leq s+1$. However, this contradicts \eqref{6.32}. % andfor  %Similar arguments holds if we take any arbitrary column say $ \omega_{\theta_1}(t_i) $ for some $ i\  [1\le i\le (m+1)] $. 
	%This completes the proof. 
	\end{proof}
\begin{lemma}\label{adding rank}
	The rank of $A_{0}$ for $\xi\neq 0$ is $(m+2)(m+1)/2$.
\end{lemma}
\begin{proof}
	From the previous lemma, we have that $\mbox{Rank}(P)\geq (m+2)(m+1)/2$. Since $A_{0}=PP^t$ and $\mbox{Rank}(P)=\mbox{Rank}(PP^t)$, we have that $\mbox{Rank}(A_{0})\geq (m+2)(m+1)/2$. However $A_{0}$ has exactly $(m+2)(m+1)/2$ rows and columns. Hence $\mbox{Rank}(A_{0})=(m+2)(m+1)/2$.
\end{proof}
Now going back to the proof of Lemma \ref{rank lemma}, we have that Rank$(U_{q})$ is exactly $q+1$ as well.
\begin{remark}
	In the general case of fixing a spherical coordinate system independent of the plane $H(x,\xi)$, the arguments would follow similarly as above, except that, one would need to consider linear combinations of the components $\Tc_{i}$ of the TRT $\Tc$ in the proofs above. 
\end{remark}
\section{Microlocal inversion}\label{Sect: Microlocal Inversion}
In this section, we give a relative left parametrix for the operator 
$ \T_{\g}^* \T_{\g}$. This will complete the proof of Theorem \ref{Main theorem}.\\
\begin{proof}[Proof of Theorem \ref{Main theorem}]
	%	Since the symbol matrix  $ A_0(x,\xi) $ has full rank. Therefore kernel of $ A_0(x,\xi) $ contains only null vector. $ i.e. $ if $ f\in S^m(\Rb^{n}) $ and $$ A_0(x,\xi)f=0 $$ then $$ f = 0 .$$  
	Now that ellipticity of $A_{0}(x,\xi)$ is shown, the construction of the relative left parametrix follows the arguments of \cite{Ramaseshan2004,Krishnan-Mishra}. We sketch the proof.
	
	Since $ A_0(x,\xi) $ is  a symmetric matrix of order $(m+1)(m+2)/2$, we diagonalize $ A_0(x,\xi) $ by an orthogonal matrix $\Oc$ such that 
	$$ A_0(x,\xi) = \Oc D \Oc^t ,$$
	where $  D $ is the diagonal matrix consisting  eigenvalues of $ A_0 $ and $\Oc$ is an orthogonal matrix whose columns are eigenvectors corresponding to the eigenvalues of $A_0  $. Since $ A_0 $ has full rank, all diagonal  entries in $ D $ are non-zero. Let 
	\begin{align*}
	B_0(x,\xi) = \Oc D^{-1} \Oc^t
	\end{align*}
	where $ D^{-1} $ is the inverse of $D$. % Since the symbol matrix of $ f $ is identity matrix $( Id) $ and 
	We have  $$ B_0(x,\xi) A_0(x,\xi)=\mbox{Id}.$$
	Define the matrix $ b_0 $ as 
	\begin{equation}
	b_{0}=\begin{cases}
	B_{0} \mbox{ if } (x,\xi)\in \Xi_{0},\\
	0 \mbox{ otherwise}.
	\end{cases}
	\end{equation}
	and $ \mathcal{B}_0 $ be the operator with symbol matrix $ b_0(x,\xi) $. The entries of $B_{0}(x,\xi)$ belong to the symbol of an  $I^{p,l}(\Delta,\Lambda)$ class, since the possible singularities of $ \Oc $ and $ D^{-1} $ are only on $\Sigma$.  
	 Note that away from the intersection $\Sigma$, $A_{0}$ is a symbol of order $-1$ and since $B_{0}$ is formed by inverting $A_{0}$, $B_{0}$ is the symbol of a pseudodifferential operator of order $1$ away from the intersection.  Therefore the operator $\Bc_{0}\in I^{0,1}(\Delta,\Lambda)$.  %   intersection $\Sigma$, $b_{0}(x,\xi)$ is a psedodifferential operator of order $1$  }
	
	Now the operator $\Tc_{\g}^{*}\Tc_{\g}\in I^{-1,0}(\Delta,\Lambda)$, and since the principal symbol of the composition $\Bc_{0}\Tc_{\g}^{*}\Tc_{\g}$ on $\Delta$ away from the intersection $\Delta \cap \Lambda$ is the product of the respective principal symbols by \cite{AntonianoandUhlmann}, which by construction is the identity on $\Delta$ away from $\Delta\cap \Lambda$, we have that $\mathcal{B}_0\T_{\g}^{*}{\T}_{\g} \in I^{-\frac{1}{2},\frac{1}{2}}(\Delta,\Lambda)$ using the composition calculus of Antoniano-Uhlmann; see Theorem \ref{Th2.1}.

	%Now as we know that operator corresponds to tensor field $ f\in S^m $ is of order $0$ and $ \T_{\g}^t{\T}_{\g}(= PDP^t) \in I^{-1,0}(\Delta,\Lambda)$ (see in \cite{Greenleaf-Uhlmann-Duke1989}). Using the composition calculus of Antoniano and Uhlmann \cite{AntonianoandUhlmann}, we prove that the entries of $\mathcal{B}_0\T_{\g}^t{\T}_{\g}$ lie in $I^{-\frac{1}{2},\frac{1}{2}}(\Delta,\Lambda)$(see in \cite{Guillemin-Uhlmann}). 
	
	Define $T_1 = \mathcal{B}_0\T_{\g}^{*}{\T}_{\g}-\mbox{Id}$. By construction the principal symbol of $T_{1}$ is $0$. % By $\sigma_0^{p,l}$, we denote the principal symbol as a distribution in the $I^{p,l}$ class. As we have, $\sigma_0^{-\frac{1}{2},\frac{1}{2}}(\mathcal{B}_0\T_{\g}^t{\T}_{\g})= \sigma_0^{-\frac{1}{2},\frac{1}{2}}(f) $. This implies $\sigma_0^{-\frac{1}{2},\frac{1}{2}}(T_1)=0$.
Let us recall the symbol calculus for $I^{p,l}\Delta, \Lambda)$ which is given by the following exact sequence \cite{Guillemin-Uhlmann}: 
	\begin{align*}
	0 \rightarrow I^{p, l-1}(\Delta, \Lambda) + I^{p-1, l}(\Delta, \Lambda)\rightarrow I^{p,l}(\Delta, \Lambda)\xrightarrow{\sigma_0} S^{p,l}(\Delta, \Sigma) \rightarrow 0
	\end{align*}
where $ S^{p,l}(\Delta, \Sigma)$ denotes the space of product type symbols, see \cite[Definition 2.3]{Krishnan-Mishra}.	With the help of this exact sequence, we decompose $T_{1}$ as %$T_1 = \mathcal{B}_0 \T_{\g}^{*}{\T}_{\g}-\mbox{Id}$  as 
	$T_{1}= T_{11}+T_{12}$ where $T_{11}\in I^{-\frac{3}{2},\frac{1}{2}}$ and $T_{12} \in I^{-\frac{1}{2},-\frac{1}{2}}$.
	
	Since $A_{0}$ has full rank, we can find two matrices $t_{11}$ and $t_{12}$ such that the principal symbol $\sigma_0(T_{1j}) = t_{1j}A_0$ for $j =1,2$. 
	\par Let $\mathcal{B}_{11}$ and $\mathcal{B}_{12}$ be the operators having symbol matrices $-t_{11}$ and $-t_{12}$ respectively. For $\mathcal{B}_{1} = \mathcal{B}_{11}+\mathcal{B}_{12}$, define 
	$T_2=  (\mathcal{B}_0+\mathcal{B}_1)\T_{\g}^{*} \T_{\g}-\mbox{Id}$. We have 
	\begin{align*}
	T_2&=  (\mathcal{B}_0+\mathcal{B}_1)\T_{\g}^{*} \T_{\g}-\mbox{Id}\\
	&= \mathcal{B}_{11}\T_{\g}^{*} \T_{\g}+\mathcal{B}_{12}\T_{\g}^{*} \T_{\g}+\mathcal{B}_0 \T_{\g}^{*} \T_{\g}-\mbox{Id}\\
	&= \underbrace{\mathcal{B}_{11}\T_{\g}^{*} \T_{\g}+T_{11}}_{K_1} +\underbrace{\mathcal{B}_{12}\T_{\g}^{*} \T_{\g} +T_{12}}_{K_2}.
	\end{align*}
	In the above expression $K_1\in I^{-\frac{3}{2},\frac{1}{2}}$ and $K_2 \in I^{-\frac{1}{2},-\frac{1}{2}}$. Also, by construction, $\sigma_0(K_1)=0$ and $\sigma_0(K_2)=0$ because $\sigma_0(\mathcal{B}_{11}\T_{\g}^{*} \T_{\g})= -\sigma_0(T_{11})$ and $\sigma_0(\mathcal{B}_{12}\T_{\g}^{*} \T_{\g})= -\sigma_0(T_{12})$. Therefore we can again use the exact sequence to decompose $K_1$ and $K_2$ as follows:
	\begin{align*}
	K_1 &= K_{11} + K_{12}, \quad \mbox{ with } K_{11} \in I^{-\frac{5}{2}, \frac{1}{2}} , K_{12} \in I^{-\frac{3}{2}, -\frac{1}{2}}\\
	K_2 &= K_{21} + K_{22}, \quad \mbox{ with } K_{21} \in I^{-\frac{3}{2}, -\frac{1}{2}} , K_{22} \in I^{-\frac{1}{2}, -\frac{3}{2}}.
	\end{align*}
	Putting this in $T_2$, we get 
	\begin{align*}
	T_2 = \underbrace{K_{11}}_{T_{20}} + \underbrace{K_{12}+K_{21}}_{T_{21}}+\underbrace{K_{22}}_{T_{22}}
	\end{align*}
	where $T_{20} \in I^{-\frac{5}{2}, \frac{1}{2}}$, $T_{21} \in  I^{-\frac{3}{2}, -\frac{1}{2}}$, $T_{22} \in I^{-\frac{1}{2}, -\frac{3}{2}}$. Therefore
	\[ T_2 \in \sum_{j=0}^{2}I^{-\frac{1}{2}-2+j, \frac{1}{2}-j}.
	\]
	Proceeding recursively, we get a sequence of operators 
	\[
	T_N \in \sum_{j=0}^N I^{-\frac{1}{2}-N+j,\frac{1}{2}-j}.
	\]
	We write this as
	\[ 
	T_N\in \sum_{j=0}^{[\frac{N}{2}]}I^{-\frac{1}{2}-N+j,\frac{1}{2}-j}+ 
	\sum_{j=[\frac{N}{2}]+1}^NI^{-\frac{1}{2}-N+j,\frac{1}{2}-j}.
	\]
	In the first sum $-\frac{1}{2}-N+j\leq -\frac{1}{2}-N+\left[\frac{N}{2}\right]$ and $\frac{1}{2}-j \leq \frac{1}{2}$. Similarly in the second sum, $-\frac{1}{2}-N+j\leq-\frac{1}{2}$ and $ \frac{1}{2}-j \leq -\frac{1}{2}-\left[\frac{N}{2}\right]$. Now we use 
	$ I^{p,l} \subset I^{p^\prime, l^\prime}$ for $ p \leq p^\prime$, $l \leq l^\prime$ to get 
	%Using the symbol calculus of Antoniano and Uhlmann \cite{AntonianoandUhlmann}, we get
	\[
	\sum_{j=0}^{[\frac{N}{2}]}I^{-\frac{1}{2}-N+j,\frac{1}{2}-j} \in I^{-\frac{1}{2}-N+\left[\frac{N}{2}\right],\frac{1}{2}}\text{ and } 
	\sum_{j=[\frac{N}{2}]+1}^NI^{-\frac{1}{2}-N+j,\frac{1}{2}-j} \in I^{-\frac{1}{2},-\frac{1}{2}-\left[\frac{N}{2}\right]}.
	\]
	In the limit $N\rightarrow \infty$,  the first term in the above expression is a smoothing term by the property that $\cap_pI^{p,l}(\Delta,\Lambda) \subset \mathcal{C}^\infty$ and the second term is an operator $\mathcal{A}$ in $I^{-\frac{1}{2}}(\Lambda)$ by the property $\cap_lI^{p,l}(\Delta,\Lambda) \subset I^p(\Lambda)$.
	Finally, we define $\mathcal{B} =\mathcal{B}_0+\mathcal{B}_1+\cdots$ and from the construction above, we get,
	$$ \mathcal{B}\T_{\g}^{*} \T_{\g}(f) = f+\mathcal{A}f+\mathcal{C}^\infty.$$
	This completes the proof of Theorem \ref{Main theorem}.
	\end{proof}
\section*{Acknowledgments:} 
VK was supported by US NSF grant DMS 1616564 and India SERB Matrics Grant, MTR/2017/000837.

\bibliographystyle{plain}
\bibliography{reference}

\end{document}

%% file: preamble.tex
\newtheorem{theorem}{Theorem}[]
\newtheorem{example}[theorem]{Example}
\newtheorem{lemma}[theorem]{Lemma}
\newtheorem{proposition}[theorem]{Proposition}

\newtheorem{definition}[theorem]{Definition}
\newtheorem{remark}[theorem]{Remark}

\newcommand{\cout}[1]{}

%% file: Transverse_ray_transform-3dim.bbl
\def\dbar{\leavevmode\hbox to 0pt{\hskip.2ex \accent"16\hss}d}
\begin{thebibliography}{10}

\bibitem{Abhishek-supporttheorem}
Anuj Abhishek.
\newblock Support theorem for the transverse ray transform of tensor fields of
  rank 2.
\newblock https://arxiv.org/abs/1804.03796.

\bibitem{Abhishek-Mishra}
Anuj Abhishek and Rohit~Kumar Mishra.
\newblock Support theorems and an injectivity result for integral moments of a
  symmetric m-tensor field.
\newblock {\em Journal of Fourier Analysis and Applications}, Oct 2018.

\bibitem{AntonianoandUhlmann}
Jos{\'e}~L. Antoniano and Gunther~A. Uhlmann.
\newblock A functional calculus for a class of pseudodifferential operators
  with singular symbols.
\newblock In {\em Pseudodifferential operators and applications ({N}otre
  {D}ame, {I}nd., 1984)}, volume~43 of {\em Proc. Sympos. Pure Math.}, pages
  5--16. Amer. Math. Soc., Providence, RI, 1985.

\bibitem{Boman-Quinto-Duke}
Jan Boman and Eric~Todd Quinto.
\newblock Support theorems for real-analytic {R}adon transforms.
\newblock {\em Duke Math. J.}, 55(4):943--948, 1987.

\bibitem{Boman1993}
Jan Boman and Eric~Todd Quinto.
\newblock Support theorems for radon transforms on real analytic line complexes
  in three-space.
\newblock {\em Transactions of the American Mathematical Society},
  335(2):877--890, 1993.

\bibitem{Desai-Lionheart}
Naeem~M. Desai and William R.~B. Lionheart.
\newblock An explicit reconstruction algorithm for the transverse ray transform
  of a second rank tensor field from three axis data.
\newblock {\em Inverse Problems}, 32(11):115009, 19, 2016.

\bibitem{GG}
M.~Golubitsky and V.~Guillemin.
\newblock {\em Stable mappings and their singularities}.
\newblock Springer-Verlag, New York-Heidelberg, 1973.
\newblock Graduate Texts in Mathematics, Vol. 14.

\bibitem{Greenleaf-Uhlmann-Duke1989}
Allan Greenleaf and Gunther Uhlmann.
\newblock Nonlocal inversion formulas for the {X}-ray transform.
\newblock {\em Duke Math. J.}, 58(1):205--240, 1989.

\bibitem{Griesmaier2018}
Roland Griesmaier, Rohit~Kumar Mishra, and Christian Schmiedecke.
\newblock Inverse source problems for {M}axwell's equations and the windowed
  {F}ourier transform.
\newblock {\em SIAM J. Sci. Comput.}, 40(2):A1204--A1223, 2018.

\bibitem{GS}
Alain Grigis and Johannes Sj{\"o}strand.
\newblock {\em Microlocal analysis for differential operators}, volume 196 of
  {\em London Mathematical Society Lecture Note Series}.
\newblock Cambridge University Press, Cambridge, 1994.
\newblock An introduction.

\bibitem{Guillemin}
Victor Guillemin.
\newblock On some results of {G}elfand in integral geometry.
\newblock In {\em Pseudodifferential operators and applications ({N}otre
  {D}ame, {I}nd., 1984)}, volume~43 of {\em Proc. Sympos. Pure Math.}, pages
  149--155. Amer. Math. Soc., Providence, RI, 1985.

\bibitem{Guillemin-Sternberg-AJM}
Victor Guillemin and Shlomo Sternberg.
\newblock Some problems in integral geometry and some related problems in
  microlocal analysis.
\newblock {\em Amer. J. Math.}, 101(4):915--955, 1979.

\bibitem{Guillemin-Uhlmann}
Victor Guillemin and Gunther Uhlmann.
\newblock Oscillatory integrals with singular symbols.
\newblock {\em Duke Math. J.}, 48(1):251--267, 1981.

\bibitem{Holman}
Sean Holman.
\newblock Recovering a tensor on the boundary from polarization and phase
  measurements.
\newblock {\em Inverse Problems}, 25(3):035009, 11, 2009.

\bibitem{Holman2}
Sean Holman.
\newblock Generic local uniqueness and stability in polarization tomography.
\newblock {\em J. Geom. Anal.}, 23(1):229--269, 2013.

\bibitem{Katsevich2002}
Alexander Katsevich.
\newblock Microlocal analysis of an {FBP} algorithm for truncated spiral cone
  beam data.
\newblock {\em J. Fourier Anal. Appl.}, 8(5):407--425, 2002.

\bibitem{K1}
Venkateswaran~P. Krishnan.
\newblock A support theorem for the geodesic ray transform on functions.
\newblock {\em J. Fourier Anal. Appl.}, 15(4):515--520, 2009.

\bibitem{Krishnan-Monard-Mishra}
Venkateswaran~P. Krishnan, Rohit~K. Mishra, and Fran\c{c}ois Monard.
\newblock On solenoidal-injective and injective ray transforms of tensor fields
  on surfaces.
\newblock {\em J. Inverse Ill-Posed Probl.}, 27(4):527--538, 2019.

\bibitem{Krishnan-Mishra}
Venkateswaran~P. Krishnan and Rohit~Kumar Mishra.
\newblock Microlocal analysis of a restricted ray transform on symmetric
  {$m$}-tensor fields in {$\Bbb{R}^n$}.
\newblock {\em SIAM J. Math. Anal.}, 50(6):6230--6254, 2018.

\bibitem{LanThesis}
Ih-Ren Lan.
\newblock On an operator associated to a restricted ray transform, 1999.
\newblock Thesis, Oregon State University.

\bibitem{Lan2003}
Ih-Ren Lan, David~V Finch, and Gunther Uhlmann.
\newblock Microlocal analysis of the x-ray transform with sources on a curve.
\newblock {\em Inside Out, Inverse Problems and Applications}, 2003.

\bibitem{Lionheart-Withers}
W.~R.~B. Lionheart and P.~J. Withers.
\newblock Diffraction tomography of strain.
\newblock {\em Inverse Problems}, 31(4):045005, 17, 2015.

\bibitem{Lionheart-Sharafutdinov}
William Lionheart and Vladimir Sharafutdinov.
\newblock Reconstruction algorithm for the linearized polarization tomography
  problem with incomplete data.
\newblock In {\em Imaging microstructures}, volume 494 of {\em Contemp. Math.},
  pages 137--159. Amer. Math. Soc., Providence, RI, 2009.

\bibitem{Lionheart_2020}
William R.~B. Lionheart.
\newblock Histogram tomography.
\newblock {\em Math. Eng.}, 2(1):55--74, 2020.

\bibitem{Melrose-Uhlmann}
R.~B. Melrose and G.~A. Uhlmann.
\newblock Lagrangian intersection and the {C}auchy problem.
\newblock {\em Comm. Pure Appl. Math.}, 32(4):483--519, 1979.

\bibitem{Novikov-Sharafutdinov}
Roman Novikov and Vladimir Sharafutdinov.
\newblock On the problem of polarization tomography. {I}.
\newblock {\em Inverse Problems}, 23(3):1229--1257, 2007.

\bibitem{Ramaseshan2004}
Karthik Ramaseshan.
\newblock Microlocal analysis of the {D}oppler transform on {$\Bbb R^3$}.
\newblock {\em J. Fourier Anal. Appl.}, 10(1):73--82, 2004.

\bibitem{Sharafutdinov_Book}
V.~A. Sharafutdinov.
\newblock {\em Integral geometry of tensor fields}.
\newblock Inverse and Ill-posed Problems Series. VSP, Utrecht, 1994.

\bibitem{Sharafutdinov_PolarizationII}
Vladimir Sharafutdinov.
\newblock The problem of polarization tomography. {II}.
\newblock {\em Inverse Problems}, 24(3):035010, 21, 2008.

\bibitem{SSU}
Vladimir Sharafutdinov, Michal Skokan, and Gunther Uhlmann.
\newblock Regularity of ghosts in tensor tomography.
\newblock {\em J. Geom. Anal.}, 15(3):499--542, 2005.

\bibitem{SU1}
Plamen Stefanov and Gunther Uhlmann.
\newblock Stability estimates for the {X}-ray transform of tensor fields and
  boundary rigidity.
\newblock {\em Duke Math. J.}, 123(3):445--467, 2004.

\bibitem{SU2}
Plamen Stefanov and Gunther Uhlmann.
\newblock Boundary rigidity and stability for generic simple metrics.
\newblock {\em J. Amer. Math. Soc.}, 18(4):975--1003 (electronic), 2005.

\bibitem{SU3}
Plamen Stefanov and Gunther Uhlmann.
\newblock Integral geometry on tensor fields on a class of non-simple
  {R}iemannian manifolds.
\newblock {\em Amer. J. Math.}, 130(1):239--268, 2008.

\bibitem{Tuy}
Heang~K. Tuy.
\newblock An inversion formula for cone-beam reconstruction.
\newblock {\em SIAM J. Appl. Math.}, 43(3):546--552, 1983.

\bibitem{Uhlmann-Vasy}
Gunther Uhlmann and Andr\'{a}s Vasy.
\newblock The inverse problem for the local geodesic ray transform.
\newblock {\em Invent. Math.}, 205(1):83--120, 2016.

\bibitem{Vertgeim}
L.~B. Vertgeim.
\newblock Integral geometry problems for symmetric tensor fields with
  incomplete data.
\newblock {\em J. Inverse Ill-Posed Probl.}, 8(3):355--364, 2000.

\end{thebibliography}
